\documentclass[10pt,a4paper]{scrartcl}%

\usepackage[T1]{fontenc}
\usepackage{epsfig}
\usepackage{graphicx}
\usepackage{amssymb}
\usepackage{amsmath}
\usepackage[utf8]{inputenc}
\usepackage{amsfonts}
\usepackage{amsthm}
\usepackage{dsfont}
\usepackage{comment}
\usepackage{enumerate}
\usepackage{mathbbol}
\usepackage{mathrsfs}
\usepackage{stmaryrd}

\usepackage{comment}

\usepackage[a4paper,vmargin={2cm,2cm},hmargin={2.25cm,2.25cm}]{geometry}

\usepackage{hyperref}
\usepackage[english]{babel}

\usepackage{listings}

\usepackage{pgf,tikz,pgfplots}

\pgfplotsset{compat=1.15}
\usetikzlibrary{arrows,shapes,positioning,calc}

\theoremstyle{plain}
\newtheorem{thm}{Theorem}
\newtheorem{lemme}[thm]{Lemma}
\newtheorem{prop}[thm]{Proposition}
\newtheorem{cor}[thm]{Corollary}

\theoremstyle{definition}

\theoremstyle{remark}

\newtheorem*{rem}{Remark}

\date{2021}

\makeindex

\begin{document}

\title{Asymptotic behaviour of the first positions of uniform parking functions}
\author{Etienne Bellin\footnote{etienne.bellin@polytechnique.edu}
\\
{\normalsize CMAP - Ecole Polytechnique}

}

\maketitle

\begin{abstract}
    In this paper we study the asymptotic behavior of a random uniform parking function $\pi_n$ of size $n$. We show that the first $k_n$ places $\pi_n(1),\dots,\pi_n(k_n)$ of $\pi_n$ are asymptotically i.i.d. and uniform on $\{1,2,\dots,n\}$, for the total variation distance when $k_n = o(\sqrt{n})$, and for the Kolmogorov distance when $k_n=o(n)$, improving results of Diaconis \& Hicks. Moreover we give bounds for the rate of convergence, as well as limit theorems for some statistics like the sum or the maximum of the first $k_n$ parking places. The main tool is a reformulation using conditioned random walks.
\end{abstract}

\section{Introduction}

A \textit{parking function} of size $n$ is a function $\pi_n:\llbracket 1,n \rrbracket \to \llbracket 1,n \rrbracket$ such that, if $\pi_n'(1) \leq \dots \leq \pi_n'(n)$ is the nondecreasing rearrangement of $(\pi_n(1),\dots,\pi_n(n))$, then $\pi_n'(i) \leq i$ for all $i$. Konheim and Weiss \cite{konheim} first introduced parking functions, in a context of information storing, to study hashing functions and they have shown that there are $(n+1)^{n-1}$ parking functions of size $n$. Since then, parking functions
became a subject of interest in the fields of combinatorics, probability, group theory and computer science. More precisely, parking functions are linked to the enumerative theory of trees and forests \cite{chassaing}, to coalescent processes \cite{louchard} \cite{broutin}, to the analysis of set partitions \cite{stanley1997}, hyperplane arrangements \cite{stanley1998} \cite{shi}, polytopes \cite{pitman} \cite{chebikin} and sandpile groups \cite{cori}. Finally, the study of probabilistic properties of parking functions has recently attracted some interest \cite{diaconis} \cite{kenyon} \cite{yin}. We refer to \cite{yan} for an extensive survey. Our initial interest for parking functions comes from the study of minimal factorisations of cycles \cite{biane2002}. 

For all $n\geq 1$ consider a random parking function $(\pi_n(i))_{1 \leq i \leq n}$ chosen uniformly among all the $(n+1)^{n-1}$ possible parking functions of size $n$. For all $1 \leq k \leq n$ denote by
\begin{equation}
\label{dtv}
d_{TV}(k,n) := \sum_{i_1,\dots,i_k=1}^n \left| \mathbb{P}(\pi_n(1)=i_1,\dots,\pi_n(k)=i_k) - \frac{1}{n^k} \right|    
\end{equation}
the total variation distance between $(\pi_n(1),\dots,\pi_n(k))$ and $(U_n(1),\dots,U_n(k))$ where $(U_n(i))_{1 \leq i \leq n}$ are i.i.d. uniformly distributed in $\llbracket 1,n \rrbracket$. Diaconis and Hicks \cite[Corollary 6]{diaconis} have shown that $d_{TV}(1,n)$ tends to 0 as $n$ tends to infinity and conjectured that for any fixed $k$, $d_{TV}(k,n)$ should be a $O(k/\sqrt{n})$. In the same paper the authors studied the Kolmogorov distance
\begin{equation}
\label{dkol}
d_K(k,n) := \max_{1 \leq i_1 \dots i_k \leq n} \left| \mathbb{P}(\pi_n(1)\leq i_1,\dots,\pi_n(k) \leq i_k) - \frac{i_1\dots i_k}{n^k}  \right|    
\end{equation}
and have shown that \cite[Theorem 3]{diaconis} for $1\leq k \leq n$:
$$
d_K(k,n) = O\left( k\sqrt{\frac{\log{n}}{n}} + \frac{k^2}{n} \right).
$$
They also discuss the growth threshold of $k$ at which $d_K$ doesn't converge towards 0 anymore and find that for $k$ of order $n$ the convergence fails. We prove a stronger version of Diaconis' and Hicks' conjecture when $k$ is allowed to grow with $n$ at rate at most $\sqrt{n}$. Moreover the Kolmogorov distance converges towards 0 when $k = o(n)$, namely:
\begin{thm}
\label{main thm}
\begin{enumerate}[(i)]
    \item If $k_n = o(\sqrt{n})$ then
\begin{equation}
\label{goal}
d_{TV}(k_n,n) = O \left( \frac{k_n}{\sqrt{n}} \right).    
\end{equation}
    \item If $k_n = o(n)$ and $\sqrt{n} = o(k_n)$ then
\begin{equation}
\label{goalii}
d_K(k_n,n) = O\left( \frac{\sqrt{n}}{k_n} + \left(\frac{k_n}{n}\right)^{0.19} \right).
\end{equation}
\end{enumerate}
\end{thm}
\begin{rem}
In Theorem \ref{main thm} (ii), $\sqrt{n}$ is assumed to be a $o(k_n)$. Since the function $k \mapsto d_K(k,n)$ is nondecreasing for fixed $n$, the distance $d_K(k_n,n)$ still tends towards 0 as long as $k_n = o(n)$. Thus, sequence $a_n = n$ satisfies $d_K(k_n,a_n) \to 0$ if $k_n = o(a_n)$ and $d_K(k_n,a_n) \not\to 0$ if $a_n = O(k_n)$. It would be very interesting to identify such a sequence for $d_{TV}$ instead of $d_K$, and, in particular, to see if $d_{TV}(k_n,n) \to 0$ when $k_n=o(n)$.
\end{rem}
The main idea to prove Theorem \ref{main thm} is to express the law of $\pi_n$ in terms of a conditioned random walk (Proposition \ref{reformulating prop} below). As an application, we obtain limit theorems for the maximum and the sum of the first $k_n$ parking places. Namely we obtain the following corollary (whose proof is postponed to the last section):
\begin{cor}
\label{cor of main thm}
\begin{enumerate}[(i)]
    \item If $k_n = o(\sqrt{n})$ and $k_n \to \infty$ then the convergence
$$
\sqrt{\frac{12}{k_n}}\left(\frac{\pi_n(1)+\dots+\pi_n(k_n)}{n}-\frac{k_n}{2}\right) \longrightarrow \mathcal{N}(0,1)
$$
holds in distribution where $\mathcal{N}(0,1)$ is a standard normal distribution.
    \item If $k_n = o(n)$ and $k_n \to \infty$ then the convergence
$$
k_n \left( 1 - \frac{1}{n}\max \{ \pi_n(1),\dots,\pi_n(k_n) \} \right) \longrightarrow \mathcal{E}(1)
$$
holds in distribution where $\mathcal{E}(1)$ is an exponential distribution with mean 1.
\end{enumerate}
\end{cor}
\begin{rem}
The complete sum $\pi_n(1)+\dots+\pi_n(n)$ has been studied and converges, after renormalization, towards a more complicated distribution involving zeros of the Airy function (see \cite[Theorem 14]{diaconis}).
\end{rem}

When $k_n \sim cn$ we obtain the following limit theorem for the first $k_n$ parking places. The proof uses other techniques and Proposition \ref{reformulating prop} (or rather its proof).

\begin{prop}
\label{prop extension max}
If $k_n \sim cn$ with $c \in (0,1]$ then for all $a \in \mathbb N$ there exists a integer-valued random variable $S_a^*$ such that $0 \leq S_a^* \leq a$ almost surely and
$$
\mathbb P (n-\max\{\pi_n(1),\dots,\pi_n(k)\} \geq a) \longrightarrow \mathbb E \left[(1-c)^{a - S_a^*}\right].
$$
\end{prop}

In section \ref{sec bij} we use a bijection between parking functions and Cayley trees and use it to reformulate the law of $\pi_n$ in terms of conditioned random walks. Then in section \ref{sec conv for tv dist} we bound the moments of a conditioned random walk in order to control the probability mass function of $\pi_n$ and prove Theorem \ref{main thm} (i). In section \ref{sec conv for k dist} we prove (ii) using arguments developed in the previous sections. Finally the last section is devoted to the proof of Corollary \ref{cor of main thm} and Proposition \ref{prop extension max}.

\medskip

\textbf{In the following, $C$ denotes a constant which may vary from line to line.}

\section{Bijection between parking functions and Cayley trees} \label{sec bij}

Here the goal is to use the bijection found by Chassaing and Marckert in \cite{chassaing} between parking functions of size $n$ and \textit{Cayley trees} with $n+1$ vertices (i.e. acyclic connected graphs with $n+1$ vertices labeled from 0 to $n$). This bijection will allow to express the joint distribution of the first positions of a uniform parking function in terms of random walks. To this end, we start with some notation and definitions. Let $\mathfrak{C}_{n+1}$ be the set of Cayley trees with $n+1$ vertices labeled from 0 to $n$ where the vertex labeled 0 is distinguished from the others (we call it the \textit{root} of the tree). Also let $P_n$ be the set of parking functions of size $n$. We consider the \textit{breadth first search} on a tree $t \in \mathfrak{C}_{n+1}$ by ordering the children of each vertex of $t$ in the increasing order of their labels (thus $t$ is viewed as a plane tree) and then taking the regular breadth first search associated to the plane order (see \cite{chassaing} for a detailed definition and see figure \ref{figure breadth first search} for an example). For $t \in \mathfrak{C}_{n+1}$ and $1 \leq i \leq n$, define $r(i,t)$ to be the rank for the breadth first search of the parent of the vertex labeled $i$ in $t$. The bijection of Chassaing and Marckert is described in the following theorem.
\begin{thm}[Chassaing and Marckert]
\label{thm bij}
The map
\begin{equation}
\label{bijection}
t \mapsto (r(1,t),\dots,r(n,t)) 
\end{equation}
is a bijection between $\mathfrak{C}_{n+1}$ and $P_n$. 
\end{thm}

\begin{rem}
Chassaing and Louchard \cite{louchard} described a similar bijection using what they call the \textit{standard} order instead of the breadth first search.
\end{rem}

\begin{figure}[!h]
\begin{center}
    \begin{tikzpicture}
    [root/.style = {draw,circle,double,minimum size = 20pt,font=\small},
    vertex/.style = {draw,circle,minimum size = 25pt, font=\small}]
        \node[vertex] (1) at (0,0) {0,\textcolor{red}{1}};
        \node[vertex] (2) at (-1,1.5) {3,\textcolor{red}{2}};
        \node[vertex] (3) at (0,1.5) {4,\textcolor{red}{3}};
        \node[vertex] (4) at (1,1.5) {6,\textcolor{red}{4}};
        \node[vertex] (5) at (-2,3) {1,\textcolor{red}{5}};
        \node[vertex] (6) at (-1,3) {2,\textcolor{red}{6}};
        \node[vertex] (7) at (0,3) {5,\textcolor{red}{7}};
        \node[vertex] (8) at (1,3) {8,\textcolor{red}{8}};
        \node[vertex] (9) at (0.5,4.5) {7,\textcolor{red}{9}};
        \node[vertex] (10) at (1.5,4.5) {9,\textcolor{red}{10}};
        \draw (1) -- (2) ;
        \draw (1) -- (3);
        \draw (1) -- (4);
        \draw (2) -- (5);
        \draw (2) -- (6);
        \draw (2) -- (7);
        \draw (4) -- (8);
        \draw (8) -- (9);
        \draw (8) -- (10);
    \end{tikzpicture}
\caption{Example of a Cayley tree $t$ with 10 vertices. For every vertex, its label is represented in black on the left and its rank for the breadth first search is in red on the right. For instance, here we have $r(5,t) = \textcolor{red}{2}$. The parking function associated with this tree by (\ref{bijection}) is $(2,2,1,1,2,1,8,4,8)$.} 
\label{figure breadth first search}
\end{center}
\end{figure}
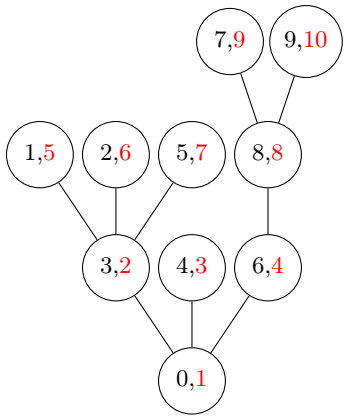
Let $(X_i)_{i \geq 1}$ be independent and identically distributed random variables distributed as a Poisson distribution of parameter 1. For all $n \geq 0$ we set $S_n := \sum_{i=1}^n (X_i -1)$ and for all $a \in \mathbb{Z}$, $\tau_{a} := \min \{ n \geq 1: S_n=a \}$ the first time that the random walk $(S_n)_n$ reaches $a$. Consider the probability measure $\mathbb{P}_n := \mathbb{P} (~\cdot~ |\tau_{-1}=n+1)$ and set $\mathbb{E}_n := \mathbb{E} [~\cdot~|\tau_{-1}=n+1]$. It is well known that a Bienaymé-Galton-Watson tree with a critical Poisson reproduction law conditioned on having $n$ vertices has the same distribution, when we uniformly randomly label the vertices from 1 to $n$, as a uniform Cayley tree with $n$ vertices (see e.g. \cite[Example\,10.2]{simplyjanson}). From this, Chassaing and Marckert deduce the following Corollary.
\begin{cor}[Chassaing and Marckert]
\label{corner stone}
Let $T_{n+1}$ be a random Cayley tree in $\mathfrak{C}_{n+1}$ with uniform distribution. The random vector $\left( \# \{1\leq j \leq n: r(j,T_{n+1})=i\} \right)_{1 \leq i \leq n+1}$ has the same distribution as $(X_i)_{1 \leq i \leq n+1}$ under $\mathbb{P}_n$.
\end{cor}
We are now able to state and prove the main result of this section.
\begin{prop}
\label{reformulating prop}
Fix $1 \leq k \leq n$ and $1 \leq i_1, \dots ,i_k \leq n$. Let $j_1<\dots<j_r$ be such that $\{i_1,\dots,i_k\} = \{j_1,\dots,j_r\}$ and define $m_s = \# \{u: i_u = j_s \}$ for all $1 \leq s \leq r$. Then
\begin{equation}
\label{reformulating}
\mathbb{P}(\pi_n(1)=i_1,\dots,\pi_n(k)=i_k) =  \frac{(n-k)!}{n!} \mathbb{E}_n \left[ \prod_{s=1}^r (X_{j_s})_{m_s} \right]    
\end{equation}
where $(x)_m := x(x-1)\cdots(x-m+1)$.
\end{prop}

\begin{proof}
Let $T_{n+1}$ be a random Cayley tree in $\mathfrak{C}_{n+1}$ with uniform distribution. Denote by $\mathfrak{S}(k,n)$ the set of all injections between $\llbracket 1,k \rrbracket$ and $\llbracket 1,n \rrbracket$. We have
\begin{align*}
\mathbb{P}(\pi_n(1)=i_1,\dots,\pi_n(k)=i_k) & = \frac{(n-k)!}{n!} \sum_{\sigma \in \mathfrak{S}(k,n)} \mathbb{P}(\pi_n(\sigma(1))=i_1,\dots,\pi_n(\sigma(k))=i_k) \\
& = \frac{(n-k)!}{n!} \mathbb{E} \left[ \sum_{\sigma \in \mathfrak{S}(k,n)} \mathds{1}_{r(\sigma(1),T_{n+1})=i_1,\dots,r(\sigma(k),T_{n+1})=i_k} \right] \\
& = \frac{(n-k)!}{n!} \mathbb{E} \left[ \prod_{s=1}^r (\# \{1\leq j \leq n: r(j,T_{n+1})=j_s\})_{m_s} \right] \\
& = \frac{(n-k)!}{n!} \mathbb{E} \left[ \prod_{s=1}^r (X_{j_s})_{m_s} \right].
\end{align*}
The first equality comes from the fact that any permutation of a parking function is still a parking function, thus any permutation induces a bijection in $P_n$. The second equality comes from Theorem \ref{thm bij} and the last from Corollary \ref{corner stone}. This completes the proof.
\end{proof}

\section{Convergence for the total variation distance} \label{sec conv for tv dist}

In this section we suppose that $k_n = o(\sqrt{n})$. We will write $k$ instead of $k_n$ to make notation lighter but keep in mind that $k$ depends on $n$. The goal of this section is to show item (i) of Theorem \ref{main thm}.

\subsection{Probability that the parking places are distinct} \label{sec proba distinct}

The first step is to reduce the problem to distinct parking places, in this case Equation (\ref{reformulating}) becomes easier. To this end we introduce the set of distinct indices $D_n := \{ (u_1,\dots,u_k) \in \llbracket 1,n \rrbracket^k: i \neq j \Rightarrow u_i \neq u_j \}$. We also introduce the set $G_n := \{ (u_1,\dots,u_k) \in \llbracket 1,n \rrbracket^k: \mathbb{P}(\pi_n(1)=u_1,\dots,\pi_n(k)=u_k) \geq (n-k)!/n! \}$ and the quantity
\begin{equation}
\label{dtv'}
\delta(k,n) := \sum_{\substack{(i_1,\dots,i_k) \\ \in D_n \cap G_n}} \left[\mathbb{P}(\pi_n(1)=i_1,\dots,\pi_n(k)=i_k) - \frac{(n-k)!}{n!}\right].    
\end{equation}
The next Lemma shows that the first $k$ parking places of a uniform parking function are all distinct with high probability. It also shows that if $\delta(k,n)$ is a $O(k/\sqrt{n})$ then so is $d_{TV}(k,n)$. Recall that $(U_n(i))_{1 \leq i \leq n}$ are i.i.d. uniformly distributed in $\llbracket 1,n \rrbracket$.
\begin{lemme}
\label{Igor trick}
We have
\begin{enumerate}[(i)]
\item $\mathbb{P}((U_n(1),\dots,U_n(k)) \in D_n) = 1 + O\left( \frac{k}{\sqrt{n}} \right)$,
\item $\mathbb{P}((\pi_n(1),\dots,\pi_n(k)) \in D_n) = 1 + O\left( \frac{k}{\sqrt{n}} \right)$,
\item $\delta(k,n) = O\left( \frac{k}{\sqrt{n}} \right) \implies d_{TV}(k,n) = O\left( \frac{k}{\sqrt{n}} \right)$.
\end{enumerate}
\end{lemme}

\begin{proof}
Let $\mu_n$ be the law of $(\pi_n(1),\dots,\pi_n(k))$ and $\nu_n$ the law of $(U_n(1),\dots,U_n(k))$, with support on the same finite space $E_n := \llbracket 1,n \rrbracket^k$. First we check (i). By Markov's inequality
$$
\nu_n(D_n^\mathrm{c}) = \mathbb{P}\left( \sum_{r<s} \mathds{1}_{U_n(r) = U_n(s)} \geq 1 \right) \leq \sum_{r < s} \mathbb{P} \left( U_n(r) = U_n(s) \right) = \sum_{r < s} \frac{1}{n} = \frac{1}{n} \frac{k(k-1)}{2}.
$$
Since $k=o(\sqrt{n})$ we have that
$$
\frac1n \frac{k(k-1)}{2} = O\left( \frac{k^2}{n} \right) = O\left( \frac{k}{\sqrt{n}} \right).
$$
Now we check (ii). To do so, we will use the Prüfer encoding (or a slight variant thereof) of a rooted Cayley tree $t \in \mathfrak{C}_{n}$ into a sequence $(p_1,\dots,p_{n-2}) \in \llbracket 0,n-1 \rrbracket^{n-2}$, which we now explain. For $a \in \llbracket 1,n-1 \rrbracket$, define $p(t,a)$ to be the label of the parent of the vertex labeled $a$ in $t$. Also, define $\ell(t)$ as the biggest leaf label of $t$, and $t^*$ the tree $t$ obtained after removing the leaf labeled $\ell(t)$ and its adjacent edge. Finally we define the sequence of trees, $t_1 := t$, $t_i := t_{i-1}^*$ for $2 \leq i \leq n-2$. The Prüfer encoding of $t$ is then defined as $p_i := p(t_i,\ell(t_i))$. For example, the Prüfer encoding of the tree in figure \ref{figure breadth first search} is $(8,8,6,0,3,0,3,3)$. The key property of this encoding is that it is a bijection between the sets $\mathfrak C_n$ and $\llbracket 0,n-1 \rrbracket^{n-2}$. Now, let $T_{n+1}$ be a uniform Cayley tree in $\mathfrak{C}_{n+1}$. Theorem \ref{thm bij} implies that $\mu_n(D_n)$ is equal to the probability that the vertices labeled 1 to $k$ in $T_{n+1}$ have distinct parents. Let $(v_1,\dots,v_k)$ be a random vector with uniform distribution in $D_n$ independent of $T_{n+1}$. Since the distribution of $T_{n+1}$ is invariant under permutation of the labels, the previous probability is also equal to the probability that the vertices labeled $v_1,\dots,v_k$ have distinct parents in $T_{n+1}$. Let $(p_1,\dots,p_{n-1})$ be the Prüfer encoding of the tree $T_{n+1}$. We complete this vector with $p_n := 0$ (this comes from the fact that $t_{n-2}$ has two vertices, one of them being the root labeled 0). Since $T_{n+1}$ is uniformly distributed in $\mathfrak{C}_{n+1}$, the vector $(p_1,\dots,p_{n-1})$ is uniformly distributed in $\llbracket 0,n \rrbracket^{n-1}$. From the previous discussion and the definition of the Prüfer encoding we deduce that
$$
\mu_n(D_n) = \mathbb{P} ( (p_{v_1},\dots,p_{v_k})\in D_n).
$$
Consider the event $Z_n := \{ v_1,\dots,v_k \neq n \}$. Under this event, it is easy to see that $(p_{v_1},\dots,p_{v_k})$ has the same law as $k$ i.i.d random variables uniformly distributed in $\llbracket 0,n \rrbracket$. So from (i) we have
$$
\mu_n(D_n) \geq \mathbb{P} ( (p_{v_1},\dots,p_{v_k})\in D_n \,|\, Z_n) \mathbb{P}(Z_n) = \left( 1 + O\left( \frac{k}{\sqrt{n}} \right) \right)\mathbb{P}(Z_n).
$$
To conclude, notice that $\mathbb{P}(Z_n) = 1-k/n$.

\medskip

\noindent Finally we show (iii). Assume that $\delta(k,n) = O\left( k/\sqrt{n} \right)$. For all $i_1,\dots,i_k$ denote by $\Delta_{i_1,\dots,i_k}$ the quantity $(\mathbb P (\pi_n(1)=i_1,\dots,\pi_n(k)=i_k)-(n-k)!/n!)$. Notice that $n^k(n-k)!/n!-1 = O(k/\sqrt{n})$ so
$$
d_{TV}(k,n) = \sum_{i_1,\dots,i_k=1}^n \left| \Delta_{i_1,\dots,i_k} \right| + O\left( \frac{k}{\sqrt{n}} \right).
$$
Denote by $d^+_n$ the sum of $\Delta_{i_1,\dots,i_k}$ over the indices in $G_n$ and $d^-_n$ the opposite of the sum over the indices in $E_n \setminus G_n$. We have that
$$
d^+_n - d^-_n = \sum_{i_1,\dots,i_k=1}^n \Delta_{i_1,\dots,i_k} = 1 - \frac{n^k(n-k)!}{n!} = O\left( \frac{k}{\sqrt{n}} \right).
$$
The last two equalities imply that
$$
d_{TV}(k,n) = d^+_n + d^-_n + O\left( \frac{k}{\sqrt{n}} \right) = 2d^+_n + O\left( \frac{k}{\sqrt{n}} \right).
$$
In conclusion we just need to show that $d^+_n$ is a $O(k/\sqrt n)$. Notice that
\begin{align*}
d^+_n = \delta(k,n) + \mu_n(D_n^\mathrm c \cap G_n) - \nu_n(D_n ^\mathrm c \cap G_n) \times \frac{n^k(n-k)!}{n!}.
\end{align*}
From (i), (ii) and the assumption on $\delta(k,n)$ we deduce that $d^+_n$ is indeed a $O(k/\sqrt{n})$.
\end{proof}
To prove (i) of Theorem \ref{main thm} it remains to show that $\delta(k,n) = O(k/\sqrt{n})$. This is the goal of the next three sections.

\subsection{A monotonicity argument}

In this section we bound the terms $\mathbb{E}_n \left[ X_{i_1} \cdots X_{i_k}   \right]$ that appear in (\ref{reformulating}) when $i_1,\dots,i_k$ are distinct with terms involving $\mathbb{E}_n \left[ S_{i_1} \cdots S_{i_k}   \right]$ since the latter are more manageable. More precisely, the aim of this section is to prove the following result.
\begin{prop}
\label{prop majorare minorare}
Fix $i_1,\dots,i_k \in \llbracket 1,n \rrbracket$ pairwise distinct. We have
\begin{align}
\label{majorare}
    i_1 \cdots i_k \, \mathbb{E}_n [ X_{i_1} \cdots X_{i_k} ] \leq \mathbb{E}_n [ (S_{i_1}+i_1) \cdots (S_{i_k}+i_k) ]
\end{align}
\end{prop}

To prove Proposition \ref{prop majorare minorare} we first state a really useful lemma which, put in simple terms, says that the steps of the random walk $S$ tend to decrease under $\mathbb{P}_n$. 
\begin{lemme}
\label{decroissance}
Fix $n \geq k \geq 1$ and $m_1, \dots, m_k \geq 1$. Let $1 \leq i_1 < \dots < i_k \leq n$, $1 \leq j_1 < \dots < j_k \leq n$ be such that $j_r \leq i_r$ for all $1 \leq r \leq k$. Finally let $f:\mathbb{N} \times \mathbb{N}^* \mapsto [0,\infty)$ be a nonnegative function such that $f(0,m)=0$, then
\begin{equation}
\mathbb{E}_n [ f(X_{i_1},m_1) \cdots f(X_{i_k},m_k) ] \leq \mathbb{E}_n [ f(X_{j_1},m_1) \cdots f(X_{j_k},m_k) ].    
\end{equation}
\end{lemme}

\begin{proof}[Proof of Lemma \ref{decroissance}]
Let $s := \min \{ r \geq 1: j_r < i_r \}$. We only need to treat the case where $i_r = j_r$ for all $r \neq s$ and $j := j_s = i_s-1$ (the general result can then be obtained by induction). Let $\sigma = (j~j+1) \in \mathfrak{S}_n$ be the permutation that transposes $j$ and $j+1$. Let $\mathcal{E}_n = \{ (x_1,\dots,x_n) \in \mathbb{N}^n: (x_1-1)+\dots+(x_t-1) = -1 \text{ iff } t=n \}$ and $\mathcal{E}'_n = \{ (x_1,\dots,x_n) \in \mathcal{E}_n : x_{j+1} > 0 \text{ or } (x_1-1)+\dots+(x_{j-1}-1)>0 \}$. Notice that $(x_1,\dots,x_n) \mapsto (x_{\sigma(1)},\dots,x_{\sigma(n)})$ is a bijection on $\mathcal{E}'_n$.
\begin{align}
    \nonumber
    \mathbb{E} [ f(X_{i_1},m_1) \cdots f(X_{i_k},m_k) \mathds{1}_{ \tau_{-1}=n+1}] & = \sum_{(x_1,\dots,x_n) \in \mathcal{E}_n} f(x_{i_1},m_1) \cdots f(x_{i_k},m_k) \mathbb{P} [ X_1=x_1, \dots, X_n = x_n ] \\ \nonumber
    & = \sum_{(x_1,\dots,x_n) \in \mathcal{E}'_n} f(x_{\sigma(i_1)},m_1) \cdots f(x_{\sigma(i_k)},m_k) \mathbb{P} [ X_1=x_1, \dots, X_n = x_n ] \\ \nonumber
    & + \sum_{(x_1,\dots,x_n) \in \mathcal{E}_n \setminus \mathcal{E}'_n} f(x_{i_1},m_1) \cdots f(x_{i_k},m_k) \mathbb{P} [ X_1=x_1, \dots, X_n = x_n ].
\end{align}
Notice that if $(x_1,\dots,x_n) \in \mathcal{E}_n \setminus \mathcal{E}'_n$ then $f(x_{j+1},m_{j+1}) = f(0,m_{j+1}) = 0$ ; in particular, since $f$ is nonnegative, $f(x_{i_1},m_1) \cdots f(x_{i_k},m_k) \leq f(x_{\sigma(i_1)},m_1) \cdots f(x_{\sigma(i_k)},m_k)$. Finally 
\begin{align}
    \nonumber
    \mathbb{E} [ f(X_{i_1},m_1) \cdots f(X_{i_k},m_k) \mathds{1}_{ \tau_{-1}=n+1}] & \leq \sum_{(x_1,\dots,x_n) \in \mathcal{E}_n} f(x_{\sigma(i_1)},m_1) \cdots f(x_{\sigma(i_k)},m_k) \mathbb{P} [ X_1=x_1, \dots, X_n = x_n ] \\ \nonumber
    & = \mathbb{E} [ f(X_{\sigma(i_1)},m_1) \cdots f(X_{\sigma(i_k)},m_k) \mathds{1}_{ \tau_{-1}=n+1}] \\ \nonumber
    & = \mathbb{E} [ f(X_{j_1},m_1) \cdots f(X_{j_k},m_k) \mathds{1}_{ \tau_{-1}=n+1}].
\end{align}
\end{proof}

\begin{rem}
In Lemma \ref{decroissance} we can for instance take $f(x,m) = x^m$ or $f(x,m) = (x)_m$. Notice that in Lemma \ref{decroissance} the indices $(i_r)_r$ must be pairwise distinct as well as the indices $(j_r)_r$. In the proof of Proposition \ref{prop majorare minorare} we extend the result, when $f(x,m)=x^m$, to the case where only the $(i_r)_r$ are pairwise distinct.
\end{rem}

\begin{proof}[Proof of Proposition \ref{prop majorare minorare}]
First we show the following inequality. Fix $n \geq k \geq 1$. Let $1 < i_1 < \dots < i_k \leq n$, $1 \leq j_1 \leq \dots \leq j_k \leq n$ be such that $j_r \leq i_r$ for all $1 \leq r \leq k$. Then
\begin{equation}
\label{eq preuve prop majorare minorare}
\mathbb{E}_n [ X_{i_1} \cdots X_{i_k} ] \leq \mathbb{E}_n [ X_{j_1} \cdots X_{j_k} ].
\end{equation}
To show (\ref{eq preuve prop majorare minorare}) it is actually enough to show the following result: let $J \subset \llbracket 1,n \rrbracket$ and $2 \leq i \leq n$ such that $i$ and $i-1$ do not belong to $J$. Let $m_j \geq 1$ for $j \in J$ and $m \geq 1$. Then
\begin{equation}
\label{young}
\mathbb{E}_n \left[ X_{i-1}^mX_{i} \prod_{j \in J} X_j^{m_j} \right] \leq \mathbb{E}_n \left[ X_{i-1}^{m+1} \prod_{j \in J} X_j^{m_j} \right].    
\end{equation}
Inequality (\ref{eq preuve prop majorare minorare}) can then be obtained by induction using Lemma \ref{decroissance} and (\ref{young}). By Young's inequality:
$$
X_{i-1}^mX_{i} \leq \frac{m}{m+1}X_{i-1}^{m+1} + \frac{1}{m+1}X_i^{m+1}.
$$
Combining this with Lemma \ref{decroissance} gives (\ref{young}) and concludes the proof of $(\ref{eq preuve prop majorare minorare})$. Now, using inequality (\ref{eq preuve prop majorare minorare}) we obtain
\begin{align*}
    i_1 \cdots i_k \, \mathbb{E}_n [ X_{i_1} \cdots X_{i_k} ] & \leq \sum_{\forall r ~ j_r \leq i_r} \mathbb{E}_n [ X_{j_1} \cdots X_{j_k} ] = \mathbb{E}_n [ (S_{i_1}+i_1) \cdots (S_{i_k}+i_k) ]
\end{align*}
which concludes the proof of Proposition \ref{prop majorare minorare}.

\end{proof}

\subsection{Bounding the moments of a random walk conditioned to be an excursion} \label{sec transfer bounds}

The goal of this section is to find bounds for the moments of the random walk $S$ conditioned to be an excursion. More precisely, the aim of this section is to show the following result.
\begin{prop}
\label{moments bounds}
There exists a constant $C > 0$ such that for all $n\geq2$, $1 \leq k \leq n-1$ and $d \geq 1$
\begin{enumerate}[(i)]
    \item   we have
            \begin{equation}
            \mathbb{E}_n \left[ S_k^d \right] \leq \left( C d n \right)^{d/2},
            \end{equation}
            
    \item   and
            \begin{equation}
            \mathbb{E}_n \left[ S_{n-k}^d \right] \leq \left( \frac{n}{n-k} \right)^{3/2}(Cdk)^{d/2},
            \end{equation}
            
    \item   as well as
            \begin{equation}
            \mathbb{E}_n \left[ S_k^d \right] \leq \left( \frac{n}{n-k} \right)^{3/2} \left( C d \sqrt{k} \right)^d.
            \end{equation}
\end{enumerate}
\end{prop}

\begin{rem}
Proposition \ref{moments bounds} (i) and (ii) actually hold true for any random walk $(S_n)_{n \geq 0}$ starting from 0 with i.i.d. increments $\xi_1,\xi_2,\dots$ all having the distribution $\mu$ whose support is $\mathbb{N}\cup \{-1\}$ and such that $\mu$ has mean 0 and finite variance. However, (iii) uses in addition the fact that a Poisson random walk has a sub-exponential tail (see e.g. \cite[Example 3]{wellner}), namely, for all $k \geq 1$ and $x \geq 0$:
\begin{equation}
\label{eq bound poisson tail}
\mathbb{P}(S_k \geq x) \leq \exp \left( -\frac{x^2}{2(k+x/3)} \right) \leq C \exp \left( -\frac{x}{4k} \right).    
\end{equation}
\end{rem}

To prove Proposition \ref{moments bounds} we will use the following lemma whose proof is postponed at the end of this section. This lemma is similar to the cyclic lemma in spirit but instead of conditioning the walk to be an excursion we only condition it to stay positive.

\begin{lemme}
\label{lemme presque}
Let $n \geq 1$ and $F:\mathbb{R}^n \rightarrow [0,+\infty)$ be invariant under cyclic shifts. Then
\begin{equation}
\label{presque lemme cyclique}
\mathbb{E} \left[ F(\xi_1,\dots,\xi_n) \mathds{1}_{S_1,\dots,S_n > 0} \right] \leq \frac{1}{n}\mathbb{E} \left[ F(\xi_1,\dots,\xi_n) (S_n\wedge n) \mathds{1}_{S_n > 0} \right].
\end{equation}
\end{lemme}

\begin{proof}[Proof of Proposition \ref{moments bounds}]
We recall that $C$ denotes a constant which may vary line to line. For (i), according to \cite[Eq.\,(32)]{addario} the maximum of the excursion of $S$ has a sub-Gaussian tail, namely there exist constants $C, \alpha > 0$ such that for all $n \geq 1$ and $x \geq 0$:
    $$
    \mathbb{P}_n \left( M_n \geq x   \right) \leq Ce^{-\alpha x^2/n}
    $$
    where $M_n := \max \{ S_0,\dots,S_n \}$ is the maximum of the walk $S$ on $[0,n]$. So we have
    \begin{align*}
    \mathbb{E}_n \left[ n^{-d/2} S_k^d \right] \leq \mathbb{E}_n \left[ n^{-d/2} M_n^d \right] = \int_0^\infty dx^{d-1}\mathbb{P}_n \left( M_n \geq \sqrt{n}x  \right) dx 
     \leq \int_0^\infty Cdx^{d-1} e^{-\alpha x^2}dx 
     \leq C^d d^{d/2}.
    \end{align*}
This shows (i). 

\medskip

\noindent The following computation is a common starting point to show both (ii) and (iii). Let $H(S_k,x)$ be either the indicator function $\mathds{1}_{S_k=x}$ or $\mathds{1}_{S_k\geq x}$ with $x>0$. Using the fact that $\mathbb{P} ( \tau_{-1}=n+1)$ is equivalent to a constant times $n^{-3/2}$ (see e.g. \cite[Eq.\,(10)]{legall}), then the Markov property and finally the cyclic lemma (see e.g. \cite[Section\,6.1]{pitmanbook}) we get
\begin{align}
    \nonumber
    \mathbb{E}_n \left[ H(S_k,x) \right] & \leq C \, n^{3/2} \, \mathbb{E} \left[ H(S_k,x) \, \mathds{1}_{ \tau_{-1}=n+1} \right] \\ \nonumber
    & \leq C \, n^{3/2} \, \mathbb{E} \left[ H(S_k,x) \, \mathds{1}_{S_1,\dots,S_k \geq 0} \, \mathbb{P} \left( \tau_{-1-S_k}=n-k \right) \right] \\ \nonumber
    & \leq C \, n^{3/2} \, \mathbb{E} \left[ H(S_k,x) \, \mathds{1}_{S_1,\dots,S_k \geq 0} \, \frac{1+S_k}{n-k} \, \mathbb{P}\left( S'_{n-k}=-1-S_k \right) \right]
\end{align}
where $S'$ is independent of $S$ and has the same distribution. Now we use Janson's inequality \cite[Lemma 2.1]{janson} which states that $\mathbb{P} (S_r = -m) \leq Cr^{-1/2}e^{-\lambda m^2/r}$ for all $r \geq 1$ and $m \geq 0$ to get
\begin{align}
\label{eq preuve lemme borne}
\nonumber
\mathbb{E}_n \left[ H(S_k,x) \right] & \leq  C \, n^{3/2} \, \mathbb{E} \left[ H(S_k,x) \, \mathds{1}_{S_1,\dots,S_k \geq 0} \, \frac{1+S_k}{(n-k)^{3/2}} \, e^{-\lambda\frac{(1+S_k)^2}{n-k}} \right] \\
& \leq C \left(\frac{n}{n-k}\right)^{3/2} \mathbb{E} \left[ H(S_k,x) \, \mathds{1}_{S_1,\dots,S_k \geq 0} \, S_k \, e^{-\lambda\frac{S_k^2}{n-k}} \right].
\end{align}
To prove (ii) we first  define $T_k := \max \{ 0 \leq i \leq k: S_i = 0 \}$, thus, using the Markov property, we obtain
\begin{align*}
\mathbb{P} \left( S_k = x, \, S_1,\dots,S_k \geq 0 \right) &= \sum_{i=0}^{k-1} \mathbb{P} \left( S_k = x, \, S_1,\dots,S_k \geq 0, \, T_k=i \right) \\
&= \sum_{i=0}^{k-1} \mathbb{P} \left(S_1,\dots,S_{i} \geq 0, \, S_i = 0 \right) \mathbb{P} \left(S_{k-i} = x, \, S_1,\dots,S_{k-i} > 0 \right).
\end{align*}
We apply Lemma \ref{lemme presque} and the local limit theorem (see e.g. \cite[Theorem\,4.2.1]{ibragimov}), which gives a constant $C>0$ such that $\mathbb{P}(S_{k-i}=x) \leq C(k-i)^{-1/2}$ for every $k>i$ and $x \in \mathbb Z$,
\begin{align*}
\mathbb{P} \left( S_k = x, \, S_1,\dots,S_k \geq 0 \right) &\leq \sum_{i=0}^{k-1} \mathbb{P} \left(S_1,\dots,S_{i} \geq 0, \, S_i = 0 \right) \frac{x}{k-i} \mathbb{P} (S_{k-i}=x) \\
&\leq C \, x\sum_{i=0}^{k-1} \mathbb{P} \left(S_1,\dots,S_{i} \geq 0, \, S_i = 0 \right) \frac{1}{(k-i)^{3/2}}.
\end{align*}
Notice that 
$$
\mathbb{P} \left(S_1,\dots,S_{i} \geq 0, \, S_i = 0 \right) = e \, \mathbb{P}(\tau_{-1}=i+1) \leq \frac{C}{(i+1)^{3/2}}. 
$$
So, finally we have
$$
\mathbb{P} \left( S_k = x, \, S_1,\dots,S_k \geq 0 \right) \leq C \, x\sum_{i=0}^{k-1} \frac{1}{((i+1)(k-i))^{3/2}} = \frac{C\,x}{(k+1)^{3/2}}\sum_{i=1}^{k-1} \left( \frac{1}{i+1} + \frac{1}{k-i} \right)^{3/2} \leq \frac{C\,x}{k^{3/2}}.
$$
Putting the previous inequality in (\ref{eq preuve lemme borne}) with $H(S_{k},x)=\mathds{1}_{S_{k}=x}$ and replacing $k$ by $n-k$ gives
$$
\mathbb{P}_n(S_{n-k}=x) \leq C \left[ \frac{n}{k(n-k)} \right]^{3/2} x^2 e^{-\lambda \frac{x^2}{k}}.
$$
We can now bound the $d$-th moment of $S_{n-k}$
\begin{align*}
    \mathbb{E}_n[S_{n-k}^d] \leq C \left[ \frac{n}{k(n-k)} \right]^{3/2} \int_0^\infty x^{d+2}e^{-\lambda \frac{x^2}{k}}dx &\leq C^d k^{d/2} \left[ \frac{n}{n-k} \right]^{3/2} \int_0^\infty x^{d+2}e^{-x^2}dx \\
    &\leq C^d k^{d/2} \left[ \frac{n}{n-k} \right]^{3/2} d^{d/2}.
\end{align*}
This concludes the proof of (ii).

\medskip

\noindent To prove (iii) we follow the same principle. Using the Markov property and Lemma \ref{lemme presque} we have
\begin{align*}
\mathbb{E} \left[ \mathds{1}_{S_k \geq x} \, \mathds{1}_{S_1,\dots,S_k \geq 0} \, S_k \right] &= \sum_{i=0}^{k-1} \mathbb{E} \left[ \mathds{1}_{S_k \geq x} \, \mathds{1}_{S_1,\dots,S_k \geq 0} \, S_k \, \mathds{1}_{T_k=i} \right] \\
&= \sum_{i=0}^{k-1} \mathbb{P} \left(S_1,\dots,S_{i} \geq 0, \, S_i = 0 \right) \mathbb{E} \left[\mathds{1}_{S_{k-i} \geq x} \, \mathds{1}_{S_1,\dots,S_{k-i} > 0} \, S_{k-i} \right] \\ 
&\leq \sum_{i=0}^{k-1} \mathbb{P} \left(S_1,\dots,S_{i} \geq 0, \, S_i=0 \right) \frac{1}{k-i} \mathbb{E} \left[\mathds{1}_{S_{k-i} \geq x} \, S_{k-i}^2 \right].
\end{align*}
Then, we apply the Cauchy-Schwarz inequality 
$$
\mathbb{E} \left[\mathds{1}_{S_{k-i} \geq x} \, S_{k-i}^2 \right] \leq \mathbb{P}(S_{k-i} \geq x)^{1/2} \mathbb{E} \left[ S_{k-i}^4 \right]^{1/2} \leq C(k-i)\mathbb{P}(S_{k-i} \geq x)^{1/2}.
$$
The last inequality comes from an explicit computation of the fourth central moment of a Poisson distribution. We combine the last inequality with (\ref{eq bound poisson tail}) to get
$$
\mathbb{E} \left[\mathds{1}_{S_{k-i} \geq x} \, S_{k-i}^2 \right] \leq C(k-i) e^{-\frac{x}{8(k-i)}} \leq C(k-i)e^{-\frac{x}{8k}} .
$$
Putting everything together we obtain
\begin{align*}
\mathbb{E} \left[ \mathds{1}_{S_k \geq x} \, \mathds{1}_{S_1,\dots,S_k \geq 0} \, S_k \right] &\leq C e^{-\frac{x}{8k}} \sum_{i=0}^{k-1} \mathbb{P} \left(S_1,\dots,S_{i} \geq 0, \, S_i=0 \right) \\
&\leq C e^{-\frac{x}{8k}} \sum_{i=0}^{k-1} e \, \mathbb{P} \left(\tau_{-1}=i+1\right) \leq C e^{-\frac{x}{8k}}.
\end{align*}
We combine the last inequality with (\ref{eq preuve lemme borne}) to get
\begin{align*}
    \mathbb{E}_n[k^{-d/2}S_k^d] = \int_0^\infty dx^{d-1} \mathbb{P}_n(S_k \geq \sqrt{k}x)dx
    \leq C \left[ \frac{n}{n-k} \right]^{3/2} \int_0^\infty dx^{d-1} e^{-\frac{x}{8}}dx
    \leq \left[ \frac{n}{n-k} \right]^{3/2} C^d d^d.
\end{align*}
This concludes the proof of (iii).

\end{proof}

We now prove Lemma \ref{lemme presque}

\begin{proof}[Proof of Lemma \ref{lemme presque}]
For $x = (x_1,\dots,x_n) \in \mathbb{R}^n$ and $i \in \llbracket 0,n-1 \rrbracket$ denote by $x^i$ the $i$-th cyclic permutation of $x$, namely $x^i := (x_{1+i},\dots,x_{n},x_1,\dots,x_i)$. Set $A_n := \{ (x_1,\dots,x_n) \in (\mathbb{N}\cup\{-1\})^n: \forall k \in \llbracket 1,n \rrbracket, \sum_{i=1}^k x_i > 0 \}$. We also set $\xi := (\xi_1,\dots,\xi_n)$. Then
\begin{align}
    \nonumber
    \mathbb{E} \left[ F(\xi_1,\dots,\xi_n) \mathds{1}_{S_1,\dots,S_n > 0} \right] = \mathbb{E} \left[ F(\xi) \mathds{1}_{\xi \in A_n} \right] & = \frac{1}{n} \sum_{i=0}^{n-1} \mathbb{E} \left[ F(\xi^i) \mathds{1}_{\xi^i \in A_n} \right] \\ \nonumber
    & = \frac{1}{n}  \mathbb{E} \left[ F(\xi) \sum_{i=0}^{n-1}\mathds{1}_{\xi^i \in A_n} \right] \\ \nonumber
    & \leq \frac{1}{n}  \mathbb{E} \left[ F(\xi) (S_n\wedge n) \mathds{1}_{S_n > 0} \right].
\end{align}
The inequality comes from the fact that the number of cyclic shifts of $\xi$ such that $\xi^i \in A_n$ is almost surely bounded by $(S_n\wedge n) \mathds{1}_{S_n > 0}$.  
\end{proof}

\subsection{Proof of Theorem \ref{main thm} (i)}

Recall we want to show that $d_{TV}(k,n) = O(k/\sqrt{n})$. In section \ref{sec proba distinct} we have shown that it is enough to show $\delta(k,n) = O(k/\sqrt{n})$ where the definition of $\delta(k,n)$ is given by (\ref{dtv'}). Thanks to Proposition \ref{reformulating prop} this quantity can be rewritten as
$$
\delta(k,n) = \frac{n^k(n-k)!}{n!} \int_{\Lambda_k} \left[ \mathbb{E}_n \left[ X_{\lceil nt_1 \rceil} \cdots X_{\lceil nt_k \rceil} \right] -1 \right] dt_1 \cdots dt_k
$$
where $\Lambda_k := \{(t_1,\dots,t_k) \in (0,1]^k: (\lceil nt_1 \rceil,\dots,\lceil nt_k \rceil) \in D_n \cap G_n\}$. As was already mentioned in section \ref{sec proba distinct} $n^k(n-k)!/n! = 1 + O(k/\sqrt n)$ so, for our purpose, it is sufficient to bound the integral:
$$
I(k,n) := \int_{\Lambda_k} \left[ \mathbb{E}_n \left[ X_{\lceil nt_1 \rceil} \cdots X_{\lceil nt_k \rceil} \right] -1 \right] dt_1 \cdots dt_k.
$$
Using inequality (\ref{majorare}) we obtain 
\begin{align}
\label{eq borne ikn}
    I(k,n) \leq & \int_{[0,1]^k} \left[ \mathbb{E}_n \left[ \left( \frac{S_{\lceil nt_1 \rceil}}{\lceil nt_1 \rceil} + 1 \right) \cdots \left( \frac{S_{\lceil nt_k \rceil}}{\lceil nt_1 \rceil} + 1 \right) \right] -1 \right] dt_1 \cdots dt_k.
\end{align}
Notice that Proposition \ref{moments bounds} (i) and (iii) imply that for all $0 \leq k \leq n$
$$
\mathbb{E}_n \left[ S_k^d \right] \leq \left( C d \sqrt{k} \right)^d
$$
Hölder's inequality then shows that for $0 \leq i_1,\dots,i_d \leq n$
\begin{equation}
\label{eq bound1}
\mathbb{E}_n \left[S_{i_1} \cdots S_{i_d}   \right] \leq  (Cd)^d \sqrt{i_1 \cdots i_d}.
\end{equation}
Expanding the products in (\ref{eq borne ikn}) and using (\ref{eq bound1}) gives
$$
I(k,n) \leq \sum_{d=1}^k \binom{k}{d} (Cd)^d n^{-d/2} \int_{[0,1]^d} \frac{dt_1 \cdots dt_d}{(t_1 \cdots t_d)^{1/2}} .
$$
Notice that $t \mapsto t^{-1/2}$ is integrable so
$$
I(k,n) \leq \sum_{d=1}^k \binom{k}{d} (Cd)^d n^{-d/2}.
$$
Using the bound $\binom{k}{d} \leq (ke/d)^d$,
$$
I(k,n) \leq \sum_{d=1}^k \left(Ce\frac{k}{\sqrt{n}}\right)^d.
$$
Since $k = o(\sqrt{n})$ we conclude that $I(k,n) = O(k/\sqrt{n})$.

\section{Convergence for the Kolmogorov distance} \label{sec conv for k dist}

In this section we suppose that $k_n = o(n)$ and $\sqrt{n}=o(k_n)$. We will write $k$ instead of $k_n$ to make notation lighter but keep in mind that $k$ depends on $n$. The goal of this section is to show Theorem \ref{main thm} (ii). The following Lemma allows us to replace the cumulative probability in (\ref{dkol}) with the term $\mathbb{E}_n[(S_{i_1}+i_1)\cdots(S_{i_n}+i_n)]$ which is more manageable.

\begin{lemme}
\label{lemme kol}
There is a constant $C>0$ such that
\begin{equation}
    d_K(k,n) \leq \frac{1}{n^k} \max_{1 \leq i_1\dots i_k \leq n} \left| \mathbb{E}_n [ (S_{i_1}+i_1)\cdots(S_{i_k}+i_k) ] - i_1\cdots i_k \right| + \frac{Ck}{n}.
\end{equation}
\end{lemme}

Before proving this Lemma we show how it implies Theorem \ref{main thm} (ii). We will also need the following simple Lemma which extends \cite[Eq.\,(27.5)]{billingsley} 
\begin{lemme}
\label{lemma billingsley}
Let $r \geq 1$, $w_1,\dots,w_r$ and $z_1,\dots,z_r$ be complex numbers of modulus smaller than, respectively, $a>0$ and $b>0$. Then
\begin{equation}
    \left| \prod_{i=1}^r w_i - \prod_{i=1}^r z_i \right| \leq \sum_{i=1}^r |w_i - z_i| a^{r-i} b^{i-1}.
\end{equation}
\end{lemme}

\begin{proof}[Proof of Lemma \ref{lemma billingsley}]
 The result readily follows from the identity
 $$
 \prod_{i=1}^r w_i - \prod_{i=1}^r z_i = (w_1-z_1)\prod_{i=2}^r w_i + z_1\left(\prod_{i=2}^r w_i - \prod_{i=2}^r z_i\right).
 $$
\end{proof}

\begin{proof}[Proof of Theorem \ref{main thm} (ii)]
Let $1/2 < \alpha < 1$. We define a sequence of intervals $I_1,\dots,I_{M+3}$ (depending on $n$) in the following way. 
\begin{multline*}
I_1 := \bigl[1,n-k\bigr),~~ I_2 := \bigl[n-k,n-k^\alpha\bigr),~~ I_3 := \left[n-k^\alpha,n-k^{\alpha^2}\right),~~ \dots\\ \dots,~~ I_{M+1} := \left[n-k^{\alpha^{M-1}},n-k^{\alpha^M}\right),~~ I_{M+2} := \left[n-k^{\alpha^M},n-n/k\right),~~ I_{M+3}:= \bigl[n-n/k,n\bigr]    
\end{multline*}
where $M$ is the biggest integer such that $n-k^{\alpha^M} \leq n-n/k$ so in particular $M = O(\log \log n)$. Let $1 \leq i_1,\dots,i_k \leq n$, using Lemma \ref{lemma billingsley} (with $a=b=1$ and noticing that $(S_i+i)/n$ under $\mathbb{P}_n$ is almost surely smaller than 1 for every $i$) we decompose the quantity $\mathbb{E}_n [ (S_{i_1}+i_1)\cdots(S_{i_k}+i_k) - i_1\cdots i_k]$ depending on what intervals the $i_j$'s belong to.
\begin{equation}
\label{eq separation interval}
\mathbb{E}_n \left[ \left(\frac{S_{i_1}+i_1}{n}\right)\cdots \left(\frac{S_{i_k}+i_k}{n}\right) - \frac{i_1}{n}\cdots\frac{i_k }{n} \right] \leq \sum_{m=1}^{M+3} \mathbb{E}_n \left[\prod_{i_j \in I_m} \frac{S_{i_j}+i_j}{n} - \prod_{i_j \in I_m} \frac{i_j}{n} \right].    
\end{equation}
Fix $m \in \{1,\dots,M+3\}$ and denote by $\iota_1,\dots,\iota_{r_m}$ the $i_j$'s that belong to $I_m$. If $m=1$ then by Lemma \ref{lemma billingsley} and Proposition \ref{moments bounds} (i)
\begin{align*}
\mathbb{E}_n \left[\prod_{i_j \in I_1} \frac{S_{i_j}+i_j}{n} - \prod_{i_j \in I_1} \frac{i_j}{n} \right] & \leq \frac{1}{n}\sum_{j=1}^{r_1} \mathbb{E}_n [S_{\iota_j}] \left(\frac{n-k}{n}\right)^{j-1} \\
& \leq \frac{C\sqrt{n}}{n} \frac{n}{k} = \frac{C\sqrt{n}}{k}.
\end{align*}
If $2 \leq m \leq M+2$ we follow the same principle but we use Proposition \ref{moments bounds} (ii) instead
\begin{align*}
\mathbb{E}_n \left[\prod_{i_j \in I_m} \frac{S_{i_j}+i_j}{n} - \prod_{i_j \in I_m} \frac{i_j}{n} \right] & \leq \frac{1}{n}\sum_{j=1}^{r_m} \mathbb{E}_n [S_{\iota_j}] \left(\frac{n-k^{\alpha^{m-1}}}{n}\right)^{j-1} \\
& \leq \frac{Ck^{\alpha^{m-2}/2}}{n} \frac{n}{k^{\alpha^{m-1}}} = \frac{C}{k^{(\alpha-1/2)\alpha^{m-2}}}.
\end{align*}
And finally if $m=M+3$
\begin{align*}
\mathbb{E}_n \left[\prod_{i_j \in I_{M+3}} \frac{S_{i_j}+i_j}{n} - \prod_{i_j \in I_{M+3}} \frac{i_j}{n} \right] & \leq \frac{1}{n}\sum_{j=1}^{r_{M+3}} \mathbb{E}_n [S_{\iota_j}] \\
& \leq r_{M+3}\frac{C}{n} \left(\frac{n}{k}\right)^{1/2} \leq C\frac{\log \log(n)}{(nk)^{1/2}}.
\end{align*}
Notice that $k^{\alpha^M} \geq n/k$ so for all $0 \leq m \leq M$, $k^{\alpha^{M-m}} \geq (n/k)^{\alpha^{-m}}$. Summing the preceding bounds for $2 \leq m \leq M+2$ gives
\begin{align*}
    \sum_{m=2}^{M+2} \mathbb{E}_n \left[\prod_{i_j \in I_m} \frac{S_{i_j}+i_j}{n} - \prod_{i_j \in I_m} \frac{i_j}{n} \right] &\leq C\sum_{m=0}^{M} \frac{1}{k^{(\alpha-1/2)\alpha^{m}}} = C\sum_{m=0}^{M} \frac{1}{k^{(\alpha-1/2)\alpha^{M-m}}} \leq C\sum_{m=0}^{M} \left( \frac{k}{n} \right)^{(\alpha-1/2)\alpha^{-m}} \\
    &\leq C\left( \frac{k}{n} \right)^{(\alpha-1/2)} +C\sum_{m=1}^{M} \left( \frac{k}{n} \right)^{-(\alpha-1/2)e\ln(\alpha)m} \\
    &\leq C\left( \frac{k}{n} \right)^{(\alpha-1/2)} + C\left( \frac{k}{n} \right)^{-(\alpha-1/2)e\ln(\alpha)}
\end{align*}
where we used the fact that $-e\ln(\alpha)m \leq 1/\alpha^m$ for all $m \geq 1$. If we take $\alpha -1/2 = \exp(-e^{-1}) -1/2 \simeq 0.1922$ then the last quantity is a $O(k/n)^{0.19}$. Putting everything together we finally obtain that
$$
\frac{1}{n^k} \max_{1 \leq i_1\dots i_k \leq n} \left| \mathbb{E}_n [ (S_{i_1}+i_1)\cdots(S_{i_k}+i_k) ] - i_1\cdots i_k \right| \leq C\frac{\sqrt{n}}{k} + C \left( \frac{k}{n} \right)^{0.19}+C\frac{\log \log(n)}{(nk)^{1/2}}.
$$
Combining the last display with Lemma \ref{lemme kol} gives the desired result.
\end{proof}

Now we prove Lemma \ref{lemme kol}.

\begin{proof}[Proof of Lemma \ref{lemme kol}]
Define the empirical distribution function of $\pi_n$, namely, for all $i \in \{1,\dots,n\}$: 
$$
F_n(i) := \frac{1}{n} \sum_{j=1}^n \mathds{1}_{\pi_n(j) \leq i}. 
$$
As suggested in \cite{diaconis} we can use a result of Bobkov \cite[Theorem 1.1]{bobkov}. The sequence $(\pi_n(1),\dots,\pi_n(n))$ is an exchangeable extension of $(\pi_n(1),\dots,\pi_n(k))$ meaning that the distribution of $(\pi_n(1),\dots,\pi_n(n))$ stays the same after any permutation. So by Theorem 1.1 of \cite{bobkov} we have that
$$
\max_{1 \leq i_1\dots i_k \leq n}\left| \mathbb{P}(\pi_n(1)\leq i_1,\dots,\pi_n(k)\leq i_k) - \mathbb{E}[ F_n(i_1)\cdots F_n(i_k) ] \right| \leq C\frac{k}{n}
$$
where $C$ is a universal constant. Using Proposition \ref{reformulating prop} and Corollary \ref{corner stone} we find
$$
\mathbb{E}[ F_n(i_1)\cdots F_n(i_k) ] = \frac{1}{n^k} \mathbb{E}_n [ (S_{i_1}+i_1)\cdots(S_{i_k}+i_k) ].
$$
Indeed, Proposition \ref{reformulating prop} implies that $(F_n(1),\dots,F_n(n))$ has the same distribution as $(G_n(1),\dots,G_n(n))$ with
$$
G_n(i) := \frac{1}{n}\sum_{j=1}^n\mathds{1}_{r(j,T_{n+1})\leq i}
$$
where $T_{n+1}$ is a uniform random tree of $\mathfrak{C}_{n+1}$. Then, Corollary \ref{corner stone} shows that 
$$
G_n(i) = \frac{1}{n} \# \{ 1 \leq j \leq n: r(j,T_{n+1}) \leq i \} \stackrel{(d)}{=} \frac{1}{n}(S_i+i)
$$
jointly for $i \in \{1,\dots,n\}$. This concludes the proof.

\end{proof}

\section{Sum and maximum of the first parking places}

We begin this section with the proof of Corollary \ref{cor of main thm}. Then we finish by proving Proposition \ref{prop extension max}.

\begin{proof}[Proof of Corollary \ref{cor of main thm}]
(i) Recall that $(U_n(i))_{1 \leq i \leq n}$ are i.i.d. uniformly distributed in $\llbracket 1,n \rrbracket$. By the central limit theorem, the convergence
$$
\sqrt{\frac{12}{k_n}}\left(\frac{U_n(1)+\dots+U_n(k_n)}{n} - \frac{k_n}{2}\right) \longrightarrow \mathcal{N}(0,1)
$$
holds in distribution. Using the first item of Theorem \ref{main thm} we deduce that the total variation distance between the distributions of $U_n(1)+\dots+U_n(k_n)$ and $\pi_n(1)+\dots+\pi_n(k_n)$ tends to 0. Thus the above convergence still holds when $U_n(i)$ is replaced with $\pi_n(i)$.

(ii) Let $x>0$, then
$$
\mathbb{P} \left[ k_n\left( 1-\frac1n \max \{ U_n(1),\dots,U_n(k_n) \} \right) \geq x \right] = 0 \vee \frac1n \left\lfloor n\left(1-\frac{x}{k_n}\right) \right\rfloor^{k_n} \xrightarrow[n \to \infty]{} e^{-x}.
$$
Using Theorem \ref{main thm} (ii) we deduce that the above convergence still holds when $U_n(i)$ is replaced with $\pi_n(i)$.
\end{proof}

\begin{proof}[Proof of Proposition \ref{prop extension max}]
In this proof we write $k$ instead of $k_n$ to make notation lighter. For every $a \geq 0$:
\begin{equation}
\label{eq preuve max}
\mathbb P (\pi_n(1),\dots,\pi_n(k) \leq n-a) = \frac{(n-k)!}{n!}\mathbb E_n \left[ (S_{n-a}+n-a)_k \right].    
\end{equation}
Indeed, following the same computation as in the proof of Proposition \ref{reformulating prop}, we have
\begin{align*}
\mathbb{P}(\pi_n(1),\dots,\pi_n(k) \leq n-a) & = \frac{(n-k)!}{n!} \sum_{\sigma \in \mathfrak{S}(k,n)} \mathbb{P}(\pi_n(\sigma(1)),\dots,\pi_n(\sigma(k)) \leq n-a) \\
& = \frac{(n-k)!}{n!} \mathbb{E} \left[ \sum_{\sigma \in \mathfrak{S}(k,n)} \mathds{1}_{r(\sigma(1),T_{n+1}),\dots,r(\sigma(k),T_{n+1}) \leq n-a} \right] \\
& = \frac{(n-k)!}{n!} \mathbb{E} \left[ (X_1+\cdots+X_{n-a})_k \right] \\
\end{align*}
which leads to (\ref{eq preuve max}) since $X_1+\cdots+X_{n-a} = S_{n-a}+n-a$. Let $\tau_n$ be a Bienaymé-Galton-Watson tree with a critical Poisson offspring distribution $\mu$ conditionned on having $n$ vertices and define $S^n$ the associated \L ukasiewicz path. More precisely, if $v_1,\dots,v_n$ are the vertices of $\tau_n$ ordered according to the lexicographic order (see e.g. \cite[Section 1.1]{legall}), then for all $0 \leq k \leq n$:
$$
S^n_k := \# \left\{ e : e \text{ is an edge adjacent to a vertex } v_i \text{ with } i \leq k \right\} - k.
$$
In the previous definition, $S^n_k$ is deduced from the first $k$ vertices $v_1,\dots,v_k$ but it is possible to see $S^n_k$ in terms of the last $n-k$ vertices $v_{k+1},\dots,v_n$:
$$
S^n_k = n-1-k-\# \left\{ e : e \text{ is an edge between two vertices } v_i \text{ and } v_j \text{ with } i,j > k \right\}.
$$
It is known that $S^{n+1}$ and $S$ under $\mathbb P_n$ have the same distribution. Thus, equality (\ref{eq preuve max}) can be rewritten in the following way
\begin{equation}
\label{eq bis preuve max}
\mathbb P (\pi_n(1),\dots,\pi_n(k) \leq n-a) = \frac{(n-k)!}{n!}\mathbb E \left[ (S_{n-a}^{n+1}+n-a)_k \right].    
\end{equation}
Let $\tau^*$ be the so-called Kesten's tree associated with $\mu$ (see e.g. \cite[Section 2.3]{abraham}). Denote by $\preceq$ the lexicographic order on the set of vertices of $\tau^*$. It is always possible to find a unique infinite sequence $u_1,u_2,\dots$ of distinct vertices of $\tau^*$ such that for all $i \geq 1$, $\{u : u \text{ is a vertex of } \tau^* \text{ such that } u_i \preceq u \} = \{ u_1,\dots,u_i\}$. In other words, $u_1,u_2,\dots$ are the last vertices of $\tau^*$ for the lexicographic order, which, necessarily, lay on the right of the infinite spine. Similarly to the \L ukasiewicz path we can define the quantity
$$
S_a^* := a-\# \left\{ e : e \text{ is an edge between two vertices } u_i \text{ and } u_j \text{ with } i,j \leq a+1 \right\}.
$$
It is known that $\tau_n$ converges in distribution, for the local topology, towards $\tau^*$ (see e.g. \cite[Section 3.3.5]{abraham}). Making use of Skorokhod's representation theorem, suppose that the latter convergence holds almost surely. Thus, $S^{n+1}_{n-a}$ converges almost surely towards $S_a^*$. Consequently, the convergence
$$
\frac{(n-k)!}{n!} (S_{n-a}^{n+1}+n-a)_k = \frac{(n-k)!}{(n-k+S_{n-a}^{n+1}-a)!} \frac{(n+S_{n-a}^{n+1}-a)!}{n!} \sim (n-cn)^{a-S_a^*}\frac{1}{n^{a-S_a^*}} \longrightarrow (1-c)^{a - S_a^*}
$$
holds almost surely. Since the above sequence is bounded by 1 we deduce that the convergence of the expectation holds which concludes the proof.
\end{proof}

\section*{Acknowledgements}
I am really grateful to Igor Kortchemski for useful suggestions and the careful reading of the manuscript.

\bibliographystyle{alpha}
\bibliography{biblio}

\begin{thebibliography}{ABDJ13}

\bibitem[ABDJ13]{addario}
L~Addario-Berry, L~Devroye, and S~Janson.
\newblock Sub-gaussian tail bounds for the width and height of conditioned
  galton–watson trees.
\newblock {\em Ann. Probab.}, 41(2):1072--1087, 03 2013.

\bibitem[AD20]{abraham}
R~Abraham and J-F Delmas.
\newblock An introduction to galton-watson trees and their local limits, 2020.

\bibitem[Bia02]{biane2002}
P~Biane.
\newblock Parking functions of types {A} and {B}.
\newblock {\em Electron. J. Combin.}, 9, 2002.

\bibitem[Bil12]{billingsley}
P~Billingsley.
\newblock {\em Probability and Measure}.
\newblock Wiley Series in Probability and Statistics, Wiley, 2012.

\bibitem[BM16]{broutin}
Nicolas Broutin and Jean-Francois Marckert.
\newblock {A new encoding of coalescent processes. Applications to the additive
  and multiplicative cases}.
\newblock {\em {Probability Theory and Related Fields}}, 166(1):515--552, 2016.

\bibitem[Bob05]{bobkov}
S~G Bobkov.
\newblock Generalized symmetric polynomials and an approximate de finetti
  representation.
\newblock {\em Journal of Theoretical Probability}, 18:399--412, 2005.

\bibitem[CL02]{louchard}
P~Chassaing and G~Louchard.
\newblock Phase transition for parking blocks, brownian excursion and
  coalescence.
\newblock {\em Random Structures and Algorithms}, 21:76–119, 08 2002.

\bibitem[CM01]{chassaing}
P~Chassaing and J-F Marckert.
\newblock Parking functions, empirical processes, and the width of rooted
  labeled trees.
\newblock {\em The Electronic Journal of Combinatorics [electronic only]},
  8(1):Research paper R14, 19 p., 2001.

\bibitem[CP10]{chebikin}
D~Chebikin and A~Postnikov.
\newblock Generalized parking functions, descent numbers, and chain polytopes
  of ribbon posets.
\newblock {\em Adv. in Appl. Math.}, 44:145--154, 2010.

\bibitem[CR00]{cori}
R~Cori and D~Rossin.
\newblock On the sandpile group of dual graphs.
\newblock {\em European J. Combin.}, 21:447--459, 2000.

\bibitem[DH17]{diaconis}
P~Diaconis and A~Hicks.
\newblock Probabilizing parking functions.
\newblock {\em Advances in Applied Mathematics}, 89:125--155, 2017.

\bibitem[Gal05]{legall}
J-F~Le Gall.
\newblock {Random trees and applications}.
\newblock {\em Probability Surveys}, 2(none):245 -- 311, 2005.

\bibitem[IL71]{ibragimov}
I~A Ibragimov and Y~V Linnik.
\newblock {\em Independent and stationary sequences of random variables}.
\newblock Wolters-Noordhoff, Groningen, 1971.

\bibitem[Jan06]{janson}
S~Janson.
\newblock Random cutting and records in deterministic and random trees.
\newblock {\em Random Structures Algorithms}, 29(2):139--179, 2006.

\bibitem[Jan12]{simplyjanson}
S~Janson.
\newblock {Simply generated trees, conditioned Galton–Watson trees, random
  allocations and condensation}.
\newblock {\em Probability Surveys}, 9(none):103 -- 252, 2012.

\bibitem[KW66]{konheim}
A~G Konheim and B~Weiss.
\newblock An occupancy discipline and applications.
\newblock {\em SIAM J. Appl. Math.}, 14:1266--1274, 1966.

\bibitem[KY21]{kenyon}
R~Kenyon and M~Yin.
\newblock Parking functions: From combinatorics to probability, 2021.

\bibitem[Pit06]{pitmanbook}
J~Pitman.
\newblock Combinatorial stochastic processes.
\newblock {\em Lecture Notes in Mathematics}, 1875, 2006.
\newblock Springer, New York.

\bibitem[Shi86]{shi}
J~Y Shi.
\newblock The kazhdan-lusztig cells in certain affine weyl groups.
\newblock {\em Lecture Notes in Mathematics}, 1179, 1986.
\newblock Springer-Verlag, Berlin.

\bibitem[SP02]{pitman}
R~P Stanley and J~Pitman.
\newblock A polytope related to empirical distributions, plane trees, parking
  functions, and the associahedron.
\newblock {\em Discrete Comput. Geom}, 27:603--634, 2002.

\bibitem[Sta97]{stanley1997}
R~P Stanley.
\newblock Parking functions and noncrossing partitions.
\newblock {\em Electron. J. Combin.}, 4(Research Paper 20), 1997.
\newblock 20 pp.

\bibitem[Sta98]{stanley1998}
R~P Stanley.
\newblock Hyperplane arrangements, parking functions and tree inversions.
\newblock {\em Mathematical Essays in Honor of Gian-Carlo Rota. Progr. Math.},
  161:359--375, 1998.
\newblock Birkhauser, Boston.

\bibitem[Wel18]{wellner}
J~A Wellner.
\newblock {The Cram\'er-Chernoff method and some exponential bounds}.
\newblock 2018.

\bibitem[Yan86]{yan}
C~H Yan.
\newblock Parking functions.
\newblock {\em Handbook of Enumerative Combinatorics. Discrete Math. Appl., CRC
  Press, Boca Raton, FL}, 2015:835--893, 1986.

\bibitem[Yin21]{yin}
M~Yin.
\newblock Parking functions: Interdisciplinary connections, 2021.

\end{thebibliography}

\end{document}